\documentclass{amsart}

\usepackage{amsmath,amsfonts,amsthm,amstext,amssymb,amscd,amsbsy}
\usepackage[latin1]{inputenc}

\usepackage{indentfirst} 
\usepackage[dvips]{color}
\usepackage{color}
\usepackage{epsfig}
\usepackage{hyperref}

\newtheorem{thm}{Theorem}

\newtheorem{rem}{Remark}

\newcommand{\rn}[1]{\mathbb{R}^{#1}}

\newcommand{\gotico}{\mathfrak}
\newcommand{\Ad}{\mbox{Ad}}

\newcommand{\F}{\mathbb{F}}

\def\bq{\begin{equation}}
\def\eq{\end{equation}}

\title{The Ricci flow of left invariant metrics on full flag manifold SU(3)/T from a dynamical systems point of view}
\author{Lino Grama and Ricardo M. Martins}

\begin{document}
\maketitle


\begin{abstract}
In this paper we study the behavior of the Ricci flow at infinity for the full flag manifold $SU(3)/T$ using techniques of the qualitative theory of differential equations, in special the Poincar\'e Compactification and Lyapunov exponents. We prove that there are four invariant lines for the Ricci flow equation, each one associated with a singularity corresponding to a Einstein metric. In such manifold, the bi-invariant normal metric is Einstein. Moreover, around each invariant line there is a cylinder of initial conditions such that the limit metric under the Ricci flow is the corresponding Einstein metric; in particular we obtain the convergence of left-invariant metrics to a bi-invariant metric under the Ricci flow.
\end{abstract}


\section{Introduction}

Let $M$ be a manifold and consider the Ricci flow equation, introduced by L. Hamilton in 1982 \cite{hamilton}, defined by
\begin{equation}
\label{eqnricci}
\frac{\partial g(t)}{\partial t}=-2Ric(g(t)),
\end{equation}
where $Ric(g)$ is the Ricci tensor of the Riemannian metric $g$. The solution of this equation, the so called Ricci flow, is a $1$-parameter family of metrics $g(t)$ in $M$. Intuitively, this is the heat equation for the metric $g$ (\cite{tian}).

Let $G$ be  a compact, connected and semisimple Lie group. A flag manifold is a homogeneous space $M=G/H$, where $H$ is the centralizer of a $1$-parameter subgroup of $G$, $\exp tX$, where $X$ is in the Lie algebra  $\gotico{g}$ of $G$. In a equivalent way, a flag manifold is the orbit of the adjoint action from $G$ to $\gotico{g}$ of an element $X\in \gotico{g}$ (see \cite{arva1}). In the context of homogeneous spaces, the study of objects that are invariant by the action of $G$ in $M$ is very natural, for example metrics, tensors, connections and so.

The main aim of this paper is to study the equations of the Ricci flow of left invariant metrics in flag manifolds. The equations of the Ricci flow of invariant metrics in Lie groups were studied in \cite{arteaga}. Specifically, we are interested in the equation \eqref{eqnricci} restricted to the set of the metrics $SU(3)$-invariants in the flag manifold \bq\label{Msu}M=\frac{SU(3)}{S(U(1)\times U(1) \times U(1))}.\eq It is know that flag manifolds admits an infinite number of invariant non-equivalent metrics. The Ricci tensor of invariant metrics in flag manifolds were calculated in \cite{arva2} and \cite{sakane}. 

The restriction from the set of metrics to the set of invariant metrics is also done in the case of Einstein metrics. We say that a metric $g$ is a Einstein metric if it satisfy the Einstein equation $Ric(g)=c g$, for some constant $c$. The Einstein equation is a system of PDEs, but when restricted to the set of invariant metrics becomes a algebraic non-linear system of equations (for instance, \cite{arva2} and \cite{evandro}).

The Ricci flow equation \eqref{eqnricci} is a nonlinear system of PDEs. When restricted to the set of invariant metrics, such system reduces to a autonomous nonlinear system of ODEs. Moreover, one of the main question about general Ricci flows is to study when and to where $g(t)$ converges when $t\rightarrow\infty$, denoted by $g_\infty$. Because of this, it is natural the proceed the study of the Ricci flow from a qualitative point of view, using tools from the theory of dynamical systems. A recent paper that also study the Ricci flow equation from a dynamical systems point of view is \cite{ds1}.

In the case we want to study, namely when $M$ is given by \eqref{Msu}, the Ricci flow equation \eqref{eqnricci} has a huge expression and it is, in some sense, equivalent to a cubic homogeneous system of differential equations in $\rn{3}$. This is a very wild scenario to study the behavior of the trajectories of a system of differential equations. For example, it is very hard to determine if there is some closed trajectory or some attracting set.  Nevertheless, we can study the behavior of system \eqref{eqnricci} at infinity, using a method introduced by Poincaré, the Poincaré Compactification. To determine the stability of the solutions, we use the Lyapunov exponents; we refer to  \cite{wiggins} for a detailed exposition of this technique.

\section{Invariant Metrics and Ricci Tensor Equations in Flag Manifolds}
Let $G$ be a compact, connect and semisimple Lie group and denote by $\gotico{g}$ the Lie algebra of $G$. Let $H$ be a closed  connect subgroup of $G$, with Lie algebra $\gotico{h}$, and consider the homogeneous space $G/H$. The point $o=eH$ is called the origin of the homogeneoous space. 

Since $G$ is compact and semisimple, the Cartan-Killing form of $\gotico{g}$ is nondegenerete and negative-definite and we will denote by $K$ the negative of the Cartan-Killing form. In this case, we say that the homogeneous space $G/H$ is reductive, that is, $
\gotico{g}=\gotico{h}\oplus \gotico{m}$ and $\Ad(H)\gotico{m}\subset \gotico{m}$, where $\gotico{m}=\gotico{h}^{\perp}$. The tangent space at origin $o$ is naturally identicate with $\gotico{m}$. We define the isotropy representation $j:H \rightarrow \mbox{GL}(\gotico{m})$ given by $j(h)=\Ad (h)\big|_{\gotico{m}}$ with $h\in H$. Then we can see $\gotico{m}$ how $\Ad(H)$-module.  

A Riemannian metric in $G/H$ is left-invariant (or simply invariant) if the diffeomorphism $L_a:G/H \rightarrow G/H$ given by $L_a(gH)=agH$ is an isometry for all $a\in G$. Left invariant metric are completely determined by its value at the origin $o$. 

If ${G}/{H}$ is reductive with an $\Ad(H)$-invariant decomposition $\gotico{g}=\gotico{h}\oplus\gotico{m}$, then there is a natural one-to-one correspondence between the $G$-invariant Riemannian metrics $g$ on ${G}/{H}$ and the $\Ad(H)$-invariant scalar product $B$ on $\gotico{m}$, see \cite{kob}.

In this work, we consider the full flag manifold $\F(3)={SU(3)}/{T}$, where $T=S(U(1)\times U(1) \times U(1))$ is the maximal torus of $SU(3)$. 

Given the flag manifold $\F(3)$, fix a reductive decomposition of the Lie algebra $\gotico{su}(3)$. Let $g$ be an invariant metric and $B$ the $\Ad$-invariant scalar product on $\gotico{m}$ corresponding to $g$. Then $B$ have the form $B(X,Y)=K(\Lambda X,Y)$, where the linear operator $\Lambda:\gotico{m}\rightarrow\gotico{m}$ is symmetric and positive with respect to the Cartan-Killing form of the $\gotico{su}(3)$. We will denote simply by $\Lambda$ an invariant metric $g$ on $\F(3)$.

Let $\gotico{m}=\gotico{m}_1 \oplus \ldots\oplus\gotico{m}_s $ be a decomposition of $\gotico{m}$ into irreducible non-equivalent $\Ad(H)$-submoules. Then (\cite{arva1}) $\Lambda \big|_{\gotico{m}_i}=\lambda_i\mbox{Id}\big|_{\gotico{m}_i} $ and the number of the parameters of a invariant metric in a flag manifold  is equal to the number of irreducible non-equivalent $\Ad(H)$-submoules. In case of the flag manifold $\F(3)$ this number is $3$, see \cite{arva2} and an inariant scalar product has the form $B(X,Y)=\lambda_{12} \cdot K(X,Y)\big|_{\gotico{m}_{12}}\oplus \lambda_{23} \cdot K(X,Y)\big|_{\gotico{m}_{23}} \oplus \lambda_{13} \cdot K(X,Y)\big|_{\gotico{m}_{13}} $. For more details about the decomposition of isotropy representation we recommend \cite{ale1}, \cite{arva1}.


The Ricci tensor of the an $G-$invariant metric is also an $G-$invariant tensor and are completely determined by its value at the origin $o$. The componets of the Ricci tensor determined by metric $\Lambda$ in flag manifold $\F(3)$ are computed in \cite{sakane} and \cite{arva2} and are:

\bq\label{ee3}
\left\{
\begin{array}{rcl}
r_{12}&=&\dfrac{1}{2\lambda_{12}}+\dfrac{1}{12}\left(\dfrac{\lambda_{12}}{\lambda_{13}\lambda_{23}}-\dfrac{\lambda_{13}}{\lambda_{12}\lambda_{23}}-\dfrac{\lambda_{23}}{\lambda_{12}\lambda_{13}}\right)\\
r_{13}&=&\dfrac{1}{2\lambda_{13}}+\dfrac{1}{12}\left(\dfrac{\lambda_{13}}{\lambda_{12}\lambda_{23}}-\dfrac{\lambda_{12}}{\lambda_{13}\lambda_{23}}-\dfrac{\lambda_{23}}{\lambda_{12}\lambda_{13}}\right)\\
r_{23}&=&\dfrac{1}{2\lambda_{23}}+\dfrac{1}{12}\left(\dfrac{\lambda_{23}}{\lambda_{12}\lambda_{13}}-\dfrac{\lambda_{13}}{\lambda_{23}\lambda_{12}}-\dfrac{\lambda_{12}}{\lambda_{23}\lambda_{13}}\right)

\end{array}
\right.
\eq

The Ricci flow equation system for left invariant metric on $\F(3)$ is given by \bq\label{r4originalLINO}\dot \lambda_{ij}=-2r_{ij}, \ 1\leq i<j\leq 3.\eq 

\section{Poincaré Compactification}

The study of the behavior of polynomial vector fields at infinity by means of the central projection started in 1881 with Poincaré, while working with planar vector fields. We recommend \cite{velasco} for the detailed description of this method, including a $n$-dimensional version. 

In this section we just sketch the method in three dimensions, to fix the notation. Consider the polynomial differential system \[\left\{\begin{array}{rcl}
\dot x_1&=&P_1(x_1,x_2,x_3),\\
\dot x_2&=&P_2(x_1,x_2,x_3),\\
\dot x_3&=&P_3(x_1,x_2,x_3),\\
\end{array}\right.\]
with the associated vector field $X=(P_1,P_2,P_3)$. The degree of $X$ is defined as $d=\max \{\deg(P_1),\deg(P_2),\deg(P_3)\}$.

Let \[S^3=\{y=(y_1,y_2,y_3,y_4)\in\rn{4}; ||y||=1\}\] be the unit sphere with north hemisphere $S^3_+=\{y\in S^3; y_{4}>0\}$, south hemisphere $S^3_-=\{y\in S^3; y_{4}<0\}$ and equator $S^3_0=\{y\in S^3; y_{4}=0\}$.

Consider the central projections $f_+:\rn{3}\rightarrow S^3_+$ and $f_-:\rn{3}\rightarrow S^3_-$ given by $f_+(x)=\dfrac{1}{\Delta(x)}(x_1,x_2,x_3,1)$ and $f_-(x)=-\dfrac{1}{\Delta(x)}(x_1,x_2,x_3,1),$ where $\Delta(x)=\sqrt{1+x_1^2+x_2^2+x_3^2}$. We shall use coordinates $y=(y_1,y_2,y_3,y_4)$ for a point $y\in S^3$.

As $f_+$ and $f_-$ are homeomorphisms, we can identify $\rn{3}$ with both $S^3_+$ and $S^3_-$. Also note that, for $x\in\rn{3}$, when $||x||\rightarrow\infty$, $f_+(x),f_-(x)\rightarrow S^3_0$. This allow us to identify $S^3_0$ with the infinity of $\rn{3}$. 

The maps $f_+$ and $f_-$ define two copies of $X$, $Df_+(x)X(x)$ in the north hemisphere, based on $f_+(x)$, and $Df_-(x)X(x)$ in the south hemisphere, based on $f_-(x)$. Denote by $\overline{X}$ the vector field on $S^3\setminus S^3_0=S^3_+\cup S^3_-$. In this way, $\overline{X}(y)$ is a vector field in $\rn{4}$ tangent to the set $S^3_+\cup S^3_-$. To extend $\overline{X}(y)$ to the sphere $S^3$, we define the Poincaré compactification of $X$ as\[p(X)(y)=y_{4}^{d-1}\overline{X}(y).\]

Denote by $\pi:\rn{3}\rightarrow\rn{3}$ the application $\pi(x_1,x_2,x_3)=\dfrac{1}{\Delta(x)}(x_1,x_2,x_3)$. Note that $\pi$ shrink $\rn{3}$ to its unitary ball and takes the infinity to the sphere $S^2$.


Now we give an explicit expression in local charts for the vector field $p(X)$. In chart $U_1$, $X$ write as \bq\label{chartU1}\dfrac{z_3^d}{(\Delta(z))^{d-1}}\left(-z_1P_1+P_2,-z_2P_1+P_3,-z_3P_1\right),\eq where $P_1,P_2,P_3$ are functions of $x_1=\dfrac{1}{z_3}$, $x_2=\dfrac{z_1}{z_3}$ and $x_3=\dfrac{z_2}{z_3}$. In chart $U_2$, $X$ write as \bq\label{chartU2}\dfrac{z_3^d}{(\Delta(z))^{d-1}}\left(-z_1P_2+P_1,-z_2P_2+P_3,-z_3P_2\right),\eq where $P_1,P_2,P_3$ are functions of $x_1=\dfrac{z_1}{z_3}$, $x_2=\dfrac{1}{z_3}$ and $x_3=\dfrac{z_2}{z_3}$. In chart $U_3$, $X$ write as \bq\label{chartU3}\dfrac{z_3^d}{(\Delta(z))^{d-1}}\left(-z_1P_3+P_1,-z_2P_3+P_2,-z_3P_3\right),\eq where $P_1,P_2,P_3$ are functions of $x_1=\dfrac{z_1}{z_3}$, $x_2=\dfrac{z_2}{z_3}$ and $x_3=\dfrac{1}{z_3}$.

We can avoid the use of the factor $\dfrac{1}{(\Delta(z))^{d-1}}$ in the expression of $X$. Note that the singularities at infinity have $z_3=0$.

Note that the Poincaré compactification associate $\rn{3}$ to its open ball, and the infinity of $\rn{3}$ to $S^2\subset\rn{3}$. Let us denote the space resulting from this association by $\mathbb{R}_\infty^{3}\cong \rn{3}\cup\{\infty\}$.


%


%

\section{Qualitative behavior of the Ricci flow}

The Ricci flow equation system \eqref{r4originalLINO}, is equivalent (outside the coordinate planes) to the system \[\dot \lambda_{ij}=\rho r_{ij}, \ 1\leq i<j\leq 3,\] with $\rho=\dfrac{6}{\lambda_{12}\lambda_{13}\lambda_{23}}$. This new system is the polynomial quadratic system

\bq\label{r4}
\left\{
\begin{array}{rcl}
\dot \lambda_{12}&=&6\lambda_{13}\lambda_{23}+\lambda_{12}^{2}-
\lambda_{13}^{2}-\lambda_{23}^{2}\\
\dot \lambda_{13}&=&6\lambda_{{12}}\lambda_{{23}}+\lambda_{13}^{2}-\lambda_{12}^{2}-\lambda_{23}^{2}\\
\dot \lambda_{23}&=&6\lambda_{{12}}\lambda_{{13}}+\lambda_{23}^{2}-\lambda_{13}^{2}-\lambda_{12}^{2}
\end{array}
\right.
\eq

We remark that this system just have a geometric meaning for $\lambda_{12},\lambda_{13}\lambda_{23}>0$. Note that the unique singularity of system \eqref{r4} is the origin, and this singularity appears from the multiplication of \eqref{ee3} by $\rho$. 

Our first result is the following: 

\begin{thm}\label{ft}Consider the flow $\psi_t:\rn{3}\rightarrow\rn{3}$, $t\in(-\delta,\delta)$ for some $\delta\in(0,\infty]$, associated with system \eqref{r4originalLINO} (or system \eqref{r4} with the flow restricted to the first octant). There is no $p\in\rn{3}$ such that $\psi_t(p)=p$ for all $t$. In other words, the flow associated to \eqref{r4originalLINO} has no singularities in finite time.\end{thm}

\begin{rem}Theorem \ref{ft}, in the context of flag manifolds, says that the Ricci flow of left-invariant metrics do not have singularities in finite time with an arbritraty initial condition, constrating with the Ricci flow of general metrics (see \cite{tian}).
\end{rem}

Writing \eqref{r4} in the local charts given by \eqref{chartU1}, \eqref{chartU2} and \eqref{chartU3}, we obtain three polynomial differential systems, each of them with $7$ singularities at $z_3=0$ (at infinity). Some of these singularities are in the intersection of the charts, so  collecting all of them we have a set with $10$ distinct singularities at infinity (as points of the equator of $S^3$). 

Now if we denote by $p_j'$ the point in $\pi(\rn{3})$ corresponding to $p_j$, we have that just 
\[\begin{array}{l}
p_1'=\left(1/\rho_1, (1/\rho_1)(2+2\sqrt{2}), 1/\rho_1\right)\sim (0.198756,0.959682,0.198756), \ \\
p_2'=\left(1/\rho_2,1/\rho_2,1/\rho_2\right)\sim (0.577350,0.577350,0.577350) \\
p_3'=\left(\rho_3(-1/2+\sqrt{2}/2),\rho_3(-1/2+\sqrt{2}/2),\rho_3\right)\sim(0.198756,0.198756,0.959682)\\
p_4'=\left((1/\rho_1)(2+2\sqrt{2}),1/\rho_1,1/\rho_1\right)\sim(0.959682,0.198756,0.198756)
\end{array}\] 
\noindent are in the first octant, where $1/\rho_3=\sqrt{1+2({-1/2+1/2\sqrt{2}})^2}$, $\rho_1=\sqrt{2+(2+\sqrt{2})^2}$ and $\rho_2=\sqrt{3}$. We will just consider these singularities, as system \eqref{r4} just make sense in the positive octant.


Consider $\gamma_j:[0,1]\rightarrow\pi(\rn{3})$ given by $\gamma_j(t)=tp_j'$. Note that $\pi^{-1}\circ\gamma_j:(0,1)\rightarrow\rn{3}$ is also a straight line in $\rn{3}$ passing through $p_j''$ ($p_j''$ is the point in $\rn{3}$ with the same coordinates of $p_j'$), with $(\pi^{-1}\circ\gamma_j)(t)\rightarrow\infty$ when $t\rightarrow 1$.

\begin{thm}\label{t4}The lines $\gamma_1$, $\gamma_2$, $\gamma_3$ and $\gamma_4$ are solutions for \eqref{r4} (and for \eqref{r4originalLINO}) passing through $p_1''$, $p_2''$, $p_3''$ and $p_4''$ respectively. 
\end{thm}
\begin{proof}This proof is a straightforward calculation.
\end{proof}

Note that $\tilde{\gamma}_j=\pi^{-1}\circ\gamma_j:(0,1)\rightarrow\rn{3}$ has the following behavior: $\lim_{t\rightarrow 1^-} \tilde{\gamma}_j(t)=\infty_j$ and $\lim_{t\rightarrow 0^+} \tilde{\gamma}_j(t)=0$, where $\infty_j$ is just a notation to make explicit the different directions of approaching to the infinity.

Note also that the points $p_1,p_3,p_4$ are saddles and $p_2$ is a attractor (this can be see looking at the jacobian matrix of each point in (any of) the local chart coordinates).

Theorem \ref{t4} provide us 4 solutions for \eqref{r4} (and for \eqref{r4originalLINO}). Now we study the stability of these solutions using Lyapunov exponents.

The procedure is the following:\\
\noindent (i) first we compute the expression of $X$ in the charts $U_1$, $U_2$ and $U_3$ using equations \eqref{chartU1}, \eqref{chartU2} and \eqref{chartU3};\\
\noindent (ii) then we compute the Lyapunov exponents for the point $p_j''$ in all of the charts $U_j$;\\
\noindent (iii) if all of the Lyapunov exponents of $p_j''$ are negative (in all of the charts) then there is a cylinder $C_j$ around the line $\gamma_j$, outside of a small ball in the origin, such that the solutions beginning at a point in $C_j$ goes to $p_j$ in the compactification.

Table \ref{ttt1} show the Lyapunov exponents $\lambda_1$, $\lambda_2$ and $\lambda_3$ for solutions $\gamma_1,\gamma_2,\gamma_3,\gamma_4$ in chart $U_1$. The calculations for charts $U_2$ and $U_3$ are similar.

\begin{center}
\begin{table}[h]\centering
\begin{tabular}{c|cccc} 
&$\gamma_1$&$\gamma_2$&$\gamma_3$&$\gamma_4$\\
\hline 
$\lambda_1$ & -0.254589 &  -0.245719 & -0.26002&-0.260018\\
$\lambda_2$ & -0.325068 & -0.333946& -0.31964 &-0.319638\\
$\lambda_3$ & -0.315206 & -0.315967 &  -0.31521&-0.315206\\
\hline
\end{tabular}
\caption{Lyapunov exponents for singularities at infinity in chart $U_1$.}
\label{ttt1}
\end{table}
\end{center}

For the proof of the next theorem follow from the calculations in Table \ref{ttt1}.

\begin{thm}\label{thmp1}Consider $\epsilon>0$ sufficiently small and $p=(x_0,y_0,z_0)$ with $x_0$, $y_0$ and $z_0$ positives, $x_0^2+y_0^2+z_0^2>\delta$ for some $\delta>1/2$ and $d(p,\gamma_{j})<\epsilon$ for some $j=1,2,3,4$. Then the solution $s$ of \eqref{r4} with $s(0)=p$ satisfy $d(s,\gamma_j)<\epsilon$ and terminates at $p_j$ (in the compactification).
\end{thm}


\begin{rem}The value $\delta>1/2$ in above Theorems is not the smaller possible. Due to numerical instability of system \eqref{r4} near the origin, we need to take initial values outside a ball $B(0,\delta)$, $\delta>0$. We choose $\delta=1/2$ just for avoid numeric complications when calculating the Lyapunov exponents.
\end{rem}

Now we translate Theorem \ref{thmp1} to a statement about the geometry of the flag manifold $SU(3)/T$.

\begin{thm}Consider $\epsilon>0$ sufficiently small and $p=(\lambda_{12},\lambda_{13},\lambda_{23})$ with $\lambda_{ij}>0$, $||p||>\delta$ for $\delta>1/2$ and $d(p,\gamma_{j})<\epsilon$ for some $j=1,2,3,4$. Let $g_t$ be the Ricci flow with $g_0$ the metric defined by $(\lambda_{12},\lambda_{13},\lambda_{23})$.\\
\noindent (i) If $j\in\{1,3,4\}$ then  $g_\infty$ is a Einstein metric.\\ 
\noindent (ii) if $j=2$, $g_\infty$ is a normal (Einstein) metric. In particular, if $p\notin \gamma_{2}$ then { $g_0$ is left-invariant} and $g_\infty$ is bi-invariant.\\

\end{thm}

\section*{Acknowledgments}

We would like to thanks Pedro Catuogno and Caio Negreiros for suggesting us the problem and Regilene Oliveira for some help with the Poincar\'e compactification. This research was partially supported by FAPESP grant 2007/05215-4 (Martins) and by CAPES and CNPq grants (Grama).

\end{document}